\documentclass[a4paper,12pt]{amsart}
\usepackage{amsmath,amssymb}

\newtheorem{theorem}{Theorem}[section]
\newtheorem{proposition}[theorem]{Proposition}
\newtheorem{question}[theorem]{Question}
\newtheorem{corollary}[theorem]{Corollary}
\newtheorem{lemma}[theorem]{Lemma}
\newtheorem{example}[theorem]{Example}
\newtheorem{definition}[theorem]{Definition}
\newenvironment{proof*}{\vskip 2mm\noindent {}}{\hfill $\Box$ \vskip 2mm}
\def\C{\mathbb C}
\def\D{\mathbb D}
\def\N{\mathbb N}
\def\R{\mathbb R}
\def\Z{\mathbb Z}

\def\eps{\varepsilon}

\renewcommand{\Im}{\operatorname{Im}}
\renewcommand{\Re}{\operatorname{Re}}

\def\ZTH{\mathcal Z(\Theta)}
\def\HTH{H^\infty / \Theta H^\infty}

\begin{document}

\title
{Sharp Invertibility in Quotient Algebras of $H^\infty$}

\author[A. Borichev, A. Nicolau, M. Ouna\"ies, and P.J. Thomas]{Alexander Borichev, Artur Nicolau, Myriam Ouna\"ies, and Pascal J. Thomas}

\address{A. Borichev: Aix Marseille Universit\'e\\
CNRS\\ I2M\\ 13331 Marseille\\ France}
\email{alexander.borichev@math.cnrs.fr}

\address{A. Nicolau: Departament de Matem{\`a}tiques\\
Universitat Aut\`onoma de Barcelona\\ and Centre de Recerca Matem{\`a}tica \\08193 Barcelona\\ Spain}
\email{artur.nicolau@uab.cat}

\address{M. Ouna\"ies: IRMA, Universit\'e de Strasbourg \\
France}
\email{ounaies@math.unistra.fr}

\address{P.J. Thomas: Universit\'e de Toulouse\\ UPS, INSA, UT1, UTM \\
Institut de Math\'e\-ma\-tiques de Toulouse\\
F-31062 Toulouse, France} 
\email{pascal.thomas@math.univ-toulouse.fr}

\begin{abstract}
We consider inner functions $\Theta$ with the zero set $\ZTH$ such that the quotient algebra $H^\infty / \Theta H^\infty$ satisfies the Strong Invertibility Property (SIP), 
that is for every $\varepsilon>0$ there exists $\delta>0$ such that 
the conditions $f \in H^\infty$, $\|[f]\|_{H^\infty/ \Theta H^\infty}=1$, $\inf_{\ZTH} |f| \ge 1-\delta$ imply that $[f]$ is invertible in $H^\infty / \Theta H^\infty$ and $\| 1/ [f] \|_{H^\infty/ \Theta H^\infty}\le 1+\varepsilon$. We prove that the SIP is equivalent to the maximal asymptotic growth of $\Theta $ away from its zero set. We also describe inner functions satisfying the SIP in terms of the narrowness of their sublevel sets and relate the SIP to the Weak Embedding Property introduced by P.Gorkin, R.Mortini, and N.Nikolski as well as to inner functions whose Frostman shifts are Carleson--Newman Blaschke products. We finally study divisors of inner functions satisfying the SIP. We describe geometrically the zero set of inner functions such that all its divisors satisfy the SIP. We also prove that a closed subset $E$ of
the unit circle is of finite entropy if and only if any singular inner function associated to a singular measure supported on $E$ is a divisor of an inner function satisfying the SIP.
\end{abstract}



\thanks{
The first author was supported by ANR-24-CE40-5470. 
The second author was supported in part by the Generalitat de Catalunya (grant 2021 SGR 00071), the Spanish Ministerio de Ciencia e Innovaci\'on (project PID2021-123151NB-I00) and the Spanish Research Agency through the Mar\'ia de Maeztu Program (CEX2020-001084-M)}

\maketitle


\section{Introduction}


Let $H^\infty$ be the algebra of bounded analytic functions on the unit disc $\D=\{z\in \C: |z|<1\}$ of the complex plane endowed with the supremum norm 
$\|f\|_\infty = \sup_{\D} |f|$. A function $\Theta$ in $H^\infty$ is called inner if it has radial limits of modulus $1$ at almost every point of the unit circle $\partial \D$. Given an inner function $\Theta$, the quotient algebra
$ H^\infty/ \Theta H^\infty$ is a Banach algebra with the natural norm $\|[f]\|_{\HTH} = \inf \{ \|f+g\Theta\|_\infty: g\in H^\infty\}$, which has recently attracted some attention. See for instance \cite{Bo, BNT, GM, GMN, N1, N2, NV}. 
The \emph{visible spectrum} of $ H^\infty/ \Theta H^\infty$ is $\ZTH$, the zero set of $\Theta$, in the sense that the evaluation at a zero of $\Theta$ is an obvious multiplicative character of $ H^\infty/ \Theta H^\infty$. 
Intuitively, the zeroes of $\Theta$ are the only places where $[f] \in H^\infty/ \Theta H^\infty$ takes well-defined values.  

Let  $\rho(z,w)
= \left| \dfrac{z-w}{1-\bar w z}\right|$  denote the pseudohyperbolic distance between the points $z,w \in \D$. Given an inner function $\Theta$ let us consider
\begin{equation}
    \label{eta}    
    \eta_\Theta (t):= \inf\left\{ |\Theta(z)|: z \in \D,\, \rho(z, \ZTH) \ge t \right\}, 0<t<1.
\end{equation}
This is a non-decreasing function. It is not necessarily continuous (see an elementary example in Section~\ref{dis5}).  
We define a function $\varkappa_\Theta$ pseudo-inverse to $\eta_\Theta$ by 
$\varkappa_\Theta(\lambda):= \inf\{ t : \eta_\Theta (t) > \lambda\}$, $\lambda\in (0,1)$. 

If $f\in H^\infty$ is such that the class $[f]$ is invertible in $H^\infty/ \Theta H^\infty$, then $\inf_{\ZTH} |f| >0$. Following P.Gorkin, R.Mortini, and N.Nikolski (\cite{GMN}),we say that an inner function $\Theta$ satisfies the Weak Embedding Property (WEP) if the converse implication holds for any $f \in H^\infty$, that is, if for any $f \in H^\infty$ the condition 
$\inf_{\ZTH} |f|>0$ implies that $[f]$ is invertible in $H^\infty/ \Theta H^\infty$. They proved that $\Theta $ satisfies the WEP if an only if $\eta_\Theta (t) >0$ for any $t>0$, see \cite{GMN}. In the subsequent paper \cite{NV} N.Nikolski and V.Vasyunin proved that for any $0<c<1$ there exists a Blaschke product $\Theta$ such that $\eta_\Theta (t) =0$ if and only if $0 \leq t \leq c$. In this paper we study inner functions $\Theta$ for which sharp invertibility in $H^\infty / \Theta H^\infty$ holds. 


\begin{definition}
\label{defsip}
An inner function $\Theta$ satisfies the \emph{Sharp Invertibility Property} {\rm(SIP)} if for any $\eps>0$,
there exists $\delta>0$ such that for any $f\in H^\infty$, the conditions $\|[f]\|\le 1$ and $\inf_{\ZTH} |f| \ge 1-\delta$
imply that $[f]$ is invertible in $ H^\infty/ \Theta H^\infty$ and $\|[f]^{-1}\|_{\HTH}\le 1+\eps$.
\end{definition}

As in \cite{GMN}, one can relate the SIP to the growth of  
$|\Theta|$ away from its zero set. 
The WEP and the SIP are independent properties, although analogous. Our first result shows that the behavior of $\eta_\Theta(t)$ when $t$ tends to $1$ is relevant for the SIP.

\begin{theorem}
\label{motiv}
An inner function $\Theta$ satisfies the Sharp Invertibility Property {\rm(SIP)} if and only if
it has Maximal Asymptotic Growth {\rm(MAG)}, that is, $\lim_{t\to 1} \eta_\Theta (t)=\sup_{0<t<1} \eta_\Theta (t)=1$. 
\end{theorem}

An immediate corollary of this Theorem is that, since the MAG property is clearly stable under finite products of the functions $\Theta$, so is the SIP.





Given an arc $I$ on the unit circle of length $|I|$ we denote the Carleson square based at $I$ by $Q(I) = \{z \in \D : z/|z| \in I , 1- |z| \leq |I| \}$. A positive measure $\mu$ in the unit disc is called a Carleson measure if there exists a constant $C=C(\mu)>0$ such that $\mu(Q(I)) \leq C |I|$ for any Carleson square $Q (I)$. 
Recall that a Blaschke product $\Theta$ with zeros $\{a_j\}$ 
is called \emph{Carleson--Newman}, denoted $\Theta \in (CN)$, if and only if $\sum_j (1-|a_j|) \delta_{a_j}$ is a Carleson measure,
or equivalently, if 
\[
\sup_k \sum_j (1-\rho(a_j,a_k)^2) = \sup_k \sum_j \frac{(1-|a_j|^2)(1-|a_k|^2)}{|1-a_j \bar a_k|^2} < \infty.
\]

For $a\in \D$ let $\varphi_a$ be the automorphism of $\D$ given by $\varphi_a (w) = (a-w)/(1- \overline{a}w)$. 
There is a link between inner functions $\Theta$ satisfying the SIP and the behavior of $\Theta$ under Frostman shifts $\varphi_a \circ \Theta $, $a \in \D$. 
It follows from (the proof of) \cite[Proposition 2]{Bo} that whenever $0<|a|< \sup \{ \eta_\Theta (t) : 0<t<1 \}$, the composition $\varphi_a \circ \Theta$
is a Carleson--Newman Blaschke product. By Theorem \ref{motiv}, any SIP inner function then belongs to the class
\begin{equation}
\label{classM}
\mathcal M := \left\{ \Theta \mbox{ inner }: 
\varphi_a \circ \Theta \in (CN), \text{ for any } a \in \D \setminus 
\{0\} \right\},
\end{equation}
introduced in \cite{MN}. Note that the class $\mathcal P:= (CN) \cap \mathcal M$, which is formed by Carleson--Newman Blaschke products 
all of whose Frostman shifts are again Carleson--Newman, had been studied beforehand and characterized in terms of
its zeros and behavior on the maximum ideal space of $H^\infty$, see \cite{Nc}, \cite{MN} and references therein.
\begin{definition}
We say that a set $S \subset \D$ is \emph{narrow} if there exists $0<R<1$ such that the set $S$ contains no pseudohyperbolic disc of pseudohyperbolic radius $R$.
\end{definition}

Let $\Theta$ be an inner function. It is proved in \cite{MN} that $\Theta \in \mathcal P$ if and only if the set $\{ z\in \D: |\Theta(z)|<1-\eps\}$ is narrow for any $0 < \eps < 1$, and that 
$\Theta \in \mathcal M$ if and only if the set $\{ z\in \D: \eps<|\Theta(z)|<1-\eps\}$ is narrow for any $0 < \eps < 1/2$. It is worth mentioning that the narrowness of the level sets of an inner function is also related to the membership of its derivative to Hardy or Nevanlinna spaces (\cite{IN}). Our next result says that inner functions satisfying the SIP can also be described by a narrowness condition.

\begin{theorem}
\label{narrowchar}
Let $\Theta$ be an inner function. Then $\Theta$ satisfies the {\rm SIP} if and only if 
the set $\{ z\in \D: 0<|\Theta(z)|<1-\eps\}$ is narrow for any $0< \eps < 1$.
\end{theorem}

We next present two consequences of Theorems \ref{motiv} and \ref{narrowchar}. The first one is a dichotomy which holds for inner functions in the class $\mathcal M$ and the second one collects several relations between the classes $\mathcal{P}$, $\mathcal{M}$ and the SIP. 

\begin{corollary}
\label{SIP-equiv}
Let $\Theta$ be an inner function in the class $\mathcal M$ defined in \eqref{classM}. Then either $\sup \{\eta_\Theta (t): 0<t<1 \}=1$ (which means that $\Theta$ satisfies 
the {\rm SIP}) or $\eta_\Theta (t)=0$ for any $0<t<1.$
\end{corollary}


\begin{corollary}
\label{WEP-cor}
\phantom{A}
\begin{itemize}
\item[{\rm (a)}] $\mathcal{P} \subset  \{ \text{inner functions satisfying the {\rm SIP}} \} \subset \mathcal{M}$.
\item[{\rm (b)}] Let $\Theta$ be an inner function satisfying the WEP. Then $\Theta \in \mathcal{M}$ if and only if $\Theta $ satisfies the {\rm SIP}.
\item[{\rm (c)}] The class $\mathcal{P}$ consists of Carleson--Newman Blaschke products satisfying the {\rm SIP}.
\end{itemize}
\end{corollary} 
 
Let $\Theta$ be a Blaschke product 
with zero sequence $\{a_j\}$. Consider 
\begin{equation}
\label{vcarl}
S_{t}(z):=\sum_{j:\rho(z,a_j)\ge t} ( 1-\rho(z,a_j)), \quad z \in \D 
\end{equation}
Note that $S_0 (0) = \sum (1- |a_j|)$ and $\sup \{S_0 (z): z \in \D \} < \infty$ if and only if $\Theta$ is a Carleson--Newman Blaschke product. 
For $\theta\in \R$ and $h, \delta >0$, we consider 
\[Q(\theta,h,\delta):=\{re^{i\theta'}: 0<1-r< \delta h,\ |\theta'-\theta|< h\}.
\]
Note that $Q(\theta, h, 1)$ is the usual Carleson square. 
Our next result describes the class $\mathcal P$ in terms of the behavior of $S_t (z)$, the distribution of its zeros and also in terms of the SIP. 

\begin{theorem}\label{NSC-P} 
Let $\Theta$ be a Blaschke product with zero set $\{a_k\}$. Then the following conditions are equivalent:
\begin{itemize}
\item[{\rm (a)}] $\Theta\in\mathcal P$.
\item[{\rm (b)}] The function $S_t$ defined in \eqref{vcarl} satisfies 
    $$
    \lim_{t\to 1}\sup_{z \in \D} S_t (z)=0.
    $$
\item[{\rm (c)}] The zeros $\{a_k \}$ of $\Theta$ satisfy  
     $$
     \lim_{\delta\to 0}\sup_{\theta\in \R,\,h>0} \frac{1}{h} \sum_{a_k \in Q(\theta,h,\delta)} (1- |a_k|) = 0.
     $$
\item[{\rm (d)}] Any inner function which divides $\Theta$ in $H^\infty$ has the {\rm SIP}.
\end{itemize}
\end{theorem}

An important subclass of $\mathcal P$ is that of the \emph{thin} Blaschke products, which are the Blaschke products $B$ whose zeros $\{a_k\}$ satisfy 
the relation $\lim_{k\to\infty} (1-|a_k|^2) |B'(a_k)|=1$, or equivalently, its zeros $\{ a_k\}$ satisfy the following thinness condition 
$$
\lim_{k\to\infty} \sum_{j: j\neq k} \frac{(1-|a_j|^2)(1-|a_k|^2)}{|1-a_j \bar a_k|^2} =0.
$$ 
The notion of thin sequences goes back to K.Hoffman \cite{H}, see also a paper \cite{KL} by A.Kerr-Lawson and a survey \cite{M} by R.Mortini.
These sequences were considered by C.Sundberg and T.Wolff in \cite{SW} in relation to interpolation problems in $H^\infty \cap VMOA$ and in the more general algebras ${QA}_B$. 
Further references include a 1982 paper \cite{V} by A.Volberg, and more recent \cite{GPW} and \cite{GMPW}. A sequence $\{a_j\}$ of points in $\D$ is called \emph{super-separated} if $\lim_{N\to\infty} \inf\{ \rho(a_j,a_k): j\neq k, j,k \ge N\} = 1$,
a necessary condition for being thin, but far from sufficient.
It is easy to prove that any thin Blaschke product has the MAG and hence the SIP. We will prove that any SIP Blaschke product
with a super-separated zero sequence must be thin. 
However we will provide examples of Carleson--Newman Blaschke products which do not have the SIP and of inner functions having the SIP which are not Carleson--Newman Blaschke products. 



Finally, we turn to the question of which inner functions can be divisors of functions satisfying the SIP. For instance,  a singular inner function, being zero-free, obviously never can satisfy the SIP, 
but in some cases one can multiply it by a Blaschke product and obtain a SIP function.
The analogous question for divisors of WEP functions 
(sometimes called ``WEP-able'') has been considered in \cite{Bo} and \cite{BNT}.

We obtain a result which underlines the proximity between the notions of SIP and WEP.
Recall that a compact subset  $E\subset \partial \D$ is said to be of \emph{finite entropy} (or being a Beurling--Carleson set) if and 
only if 
$$
\int_{\partial \D} \log (  \mbox{dist}(\zeta, E)^{-1}) |d\zeta| < \infty . 
$$ 
Note that compact sets of finite entropy must have Lebesgue measure zero. Given a positive Borel measure $\mu$ on $\partial \D$ which is singular with respect to Lebesgue measure, let 
\begin{equation}
\label{singdef}
S_\mu(z)= \exp \left( - \int_{\partial \D} \frac{\xi + z}{ \xi -z} d \mu (\xi) \right), \qquad z \in \mathbb D , 
\end{equation}
be the corresponding singular inner function. We then have a situation quite parallel to that of WEP-ability, see \cite{BNT}. 
 

\begin{theorem}
\label{sipentropy}
Let $E$ be a closed subset of the unit circle $\partial \D$. Then the following conditions are equivalent:
\begin{itemize}
\item[{\rm (a)}] $E$ has  finite entropy.
\item[{\rm (b)}] For any positive Borel singular measure $\mu$ supported on $E$, the corresponding singular inner function $S_\mu$ is a divisor
of an inner function having the {\rm SIP}.
\end{itemize}
\end{theorem}

Although we have some information on the relations between the WEP and the SIP, the following question remains open.

\begin{question} Is it true that any function satisfying the {\rm SIP} is a divisor of a function satisfying the {\rm WEP}\,?
\end{question}

The paper is organized as follows. Theorem \ref{motiv} is established in Section \ref{equiv}. Section \ref{level} is devoted to the proof of Theorem \ref{narrowchar} and its Corollaries. Theorem \ref{NSC-P} is proved in Section \ref{classP}. Section \ref{divsip} is devoted to the proof of Theorem \ref{sipentropy}.
Finally Section \ref{examples} is devoted to discussing the relations between thin, Carleson--Newman Blaschke products, inner functions satisfying the WEP, and inner functions satisfying the SIP. Several relevant examples are also provided. 

It is a pleasure to thank Raymond Mortini for drawing our attention to the results in \cite{GM}.

\section{Equivalence between Sharp Invertibility and Maximal Growth}
\label{equiv}

The proof of Theorem \ref{motiv} uses the following beautiful quantitative version of the Corona Theorem proved by P.Jones.

\begin{theorem}\cite[Theorem 1]{J}.
\label{jones}
Let $N \ge 2$. For any $\eps>0$, there exists $\delta>0$ such that if $f_1, \dots, f_N \in H^\infty$,
$\|f_j\|_\infty \le 1$, $1\le j \le N$, and $\max_{1\le j \le N} |f_j(z)| \ge 1-\delta$
for all $z\in \D$, then there exist $g_1, \dots, g_N \in H^\infty$ such that
\[
\sum_{j=1}^N f_j(z) g_j(z) = 1, \qquad z\in \D,
\]
and $\max_{1\le j \le N} \|g_j\|_\infty \le 1+\eps$.
\end{theorem}

\begin{proof*}{\it Proof of Theorem \ref{motiv}.} 
First we show that  MAG is sufficient for sharp invertibility. Let $f \in H^\infty$ with $\|[f]\|_{\HTH} = 1$ and $\inf_{\ZTH}|f| \ge 1- \eta$. Then for any $\delta>0$ we can find a representative $f_1$
of the class $[f]$ with $\|f_1\|_\infty \le 1+\delta$ and $\inf_{\ZTH}|f_1|\ge 1-\eta$.  
Taking $f_2:= (1+\delta)^{-1} f_1$, we have 
$\|f_2\|_\infty \le 1$ and 
\[
 \inf_{\ZTH} |g| \ge (1+\delta)^{-1} (1-\eta)
 \]
for any $g\in [f_2]$. Since $[f]^{-1} = (1+\delta)[f_2]^{-1}$, we may restrict ourselves to the case of $f \in H^\infty$, $\|f \|_\infty \leq 1$ such that $\inf_{\ZTH}|f|\ge 1-\eta$.

Now we prove that for any $ 0 < \eps_1 < 1$, there exists $0 < \eta_1 < 1$ such that  $\|f\|_{\infty}\le 1$ and
$ \inf_{\ZTH} |f| \ge 1-\eta_1$ imply that 
$$
\inf_{z \in \D} \max(|\Theta (z)|,|f (z)|) \ge 1-\eps_1 . 
$$ 
Indeed, let $0 < \delta_1 < 1$ such that $ \eta_\Theta (\delta_1) \ge 1-\eps_1$.
This means that $|\Theta(z)|\ge 1-\eps_1$ if $\rho(z, \ZTH ) \ge \delta_1$.
On the other hand, by the Schwarz--Pick Lemma, when $\rho(z, \ZTH ) \le \delta_1$, we have 
\[
|f(z)| \ge \frac{1-\eta_1 - \delta_1}{1-(1-\eta_1)\delta_1} \ge 1-\eps_1,
\]
provided that we pick $0 < \eta_1 \le (1-\delta_1)(1+\delta_1(1-\eps_1))^{-1} \eps_1$.

Since $\max \{ |\Theta (z)|, |f(z)| \} \geq 1 - \varepsilon_1$ for any $z \in \D$, Theorem~\ref{jones} 
provides functions $g, h \in H^\infty$ such that $fg+\Theta h\equiv 1$. Then we have $\|[f]^{-1}\|_{\HTH} \le \|g\|_\infty\le 1+\eps$ and it suffices to choose $\eps_1:=\delta$ in Theorem \ref{jones} and then $\eta=\eta_1$ depending on $\eps_1$,
as explained above.

Assume now that $\Theta$ enjoys the SIP. Let $\eps>0$ and pick the corresponding $\delta>0$ given by Definition \ref{defsip}. Let $z_0\in \D$ be such that
$\rho(z_0, \ZTH)\ge 1-\delta$. 
As before, for any $a\in \D$, let $\varphi_a$ be the involutive 
automorphism of the disc exchanging $0$ and $a$.  By the hypothesis applied
to $f:= \varphi_{z_0}$, there exist $g, h \in H^\infty$ such that $1<\gamma:=\|g\|_\infty \le 1+\eps$
and $\varphi_{z_0}(z) g(z) + \Theta(z) h(z) =1$ for all $z\in \D$. 

Then, let $g_1= \gamma^{-1} g$, and $g_2:= \varphi_{z_0} g_1= \gamma^{-1}(1- h \Theta)$, 
so that $\|g_2 \|_\infty = 1$,
 $g_2(z_0)=0$, and $g_2= \gamma^{-1} $ on $\ZTH$. Finally,
  $g_3 := \varphi_{\gamma^{-1}} \circ g_2$ is a function  vanishing on $\ZTH$
such that $\|g_3 \|_\infty \leq 1$. Also, 
\[
g_3= \frac{\gamma^{-1}-g_2}{1-\gamma^{-1}g_2} = \frac{\gamma^{-1} h \Theta}{1-\gamma^{-1}g_2} 
\]
is a multiple of $\Theta$, which is inner. Hence $|g_3| \leq | \Theta |$ on $\D$ and we deduce

\[
|\Theta(z_0)| \ge |g_3 (z_0)| =|\varphi_{\gamma^{-1}} (0)| = \gamma^{-1} \ge \frac1{1+\eps} > 1-\eps.
\]
\end{proof*}

The proof of the direct part of Theorem \ref{motiv}, using $N\geq 1$ in P.Jones Theorem, also yields the following version of the Corona theorem for $ H^\infty/ \Theta H^\infty$, of which the sharp invertibility property is the case $N=1$.

\begin{proposition}
\label{corona}
Let $\Theta$ be an inner function having {\rm MAG} and let $N\ge 1$ be an integer. Then for any $\eps>0$, there exists $\delta>0$ such that if $f_j \in H^\infty $,
$\|[f_j ]\|_{\HTH}\le 1$, $1\le j \le N$, and $\inf_{z\in \ZTH}\max_{1\le j \le N} |f_j(z)| \ge 1-\delta$, 
 then there exist $g_1, \dots, g_N \in H^\infty$ such that
\[
\sum_{j=1}^N f_j g_j= 1
\]
in $\HTH$ and $\max_{1\le j \le N} \|g_j\|_\infty \le 1+\eps$.
\end{proposition}

\section{Sharp Invertibility and Sublevel sets}
\label{level}

It will be convenient here and in some other places to use the hyperbolic (or Poincar\'e) distance given by 
\begin{equation}
 d_H (z,w) := \tanh^{-1} \rho(z,w)= \frac12 \log \frac{1+\rho(z,w)}{1-\rho(z,w)} , \qquad z, w \in \D . 
\end{equation}
The hyperbolic disc centered at $z \in \D$ of hyperbolic radius $R>0$ will be denoted by $D_H(z,R) := \{ w\in \D : d_H(z,w)<R\}$. 




We first establish the necessity in Theorem \ref{narrowchar}. 
\begin{proposition}
\label{SIP-M}${}$
Let $\Theta$ be an inner function that satisfies the SIP. Then for any $ 0 < \eps < 1$ there exists $ R>0$ such that the set $\{ w \in \D: 0<|\Theta(w)|<1-\eps\}$ contains no hyperbolic disc of hyperbolic radius $R$. 
\end{proposition}

\begin{proof}
Fix $0< \eps <1$. Since $\Theta$ has the SIP, there exists $0 < \delta < 1$ such that $\eta_\Theta (\delta) \ge 1 - \eps$. Pick $R>0$ such that $\tanh (R) = \delta$. 
For any $z \in \D$ we have either $|\Theta(z)| \geq 1-\eps$ or
 $D_H(z,R)\cap \ZTH\neq\emptyset$.  In both cases $D_H(z,R)\not\subset \{ w \in \D: 0<|\Theta(w)|<1-\eps\}$.
\end{proof}

For the converse, we first need an auxiliary result on positive harmonic functions.

   \begin{lemma}
\label{step}
For $2\le r_1 \le r_0-1$, denote $\rho_j:= \tanh r_j$, $j=0,1$. Let
$h$ be a positive harmonic function on $D_H(0,r_0)=D (0, \rho_0)$.
If there exists $\theta_0 \in \R$
such that $h(\rho_1 e^{i\theta_0}) \le \frac14 h(0)$, then 
there exists $\theta_1 \in \R$ such that 
$h(\rho_1 e^{i\theta_1}) \ge \left(1+\frac1{10} e^{-2r_1}\right) h(0)$.
\end{lemma}

Note that the conclusion of the lemma does not depend on the value of $r_0$
as long as $r_0\ge r_1 +1$.

\begin{proof*}{\it Proof of Lemma \ref{step}.}

Without loss of generality, assume $\theta_0=0$. 
Let $s:=\rho_1/\rho_0$. Note that $\tanh2 \le s<1$.
By  Harnack's inequality, for $|\theta|\le \delta$,
\[
h(\rho_1 e^{i\theta}) \le h(\rho_1 ) \frac{1+\rho(s,se^{i\delta})}{1-\rho(s,se^{i\delta})}.
\]
Choosing $\delta:= \cos^{-1} \left( 1- \left(\frac{1-s^2}{4s}\right)^2 \right)$,
we can check that $\rho(s,se^{i\delta}) \le \frac{1}{3}$, and so 
$h(\rho_1 e^{i\theta}) \le 2 h(\rho_1 ) \le \frac{1}{2} h(0)$ for $|\theta|\le \delta$. 

By the mean value property of harmonic functions,
\[
h(0)= \int_{-\pi}^\pi h(\rho_1 e^{i\theta}) \frac{d\theta}{2\pi}
\le \frac{2\delta}{2\pi} \frac{1}{2} h(0) + \frac{2\pi-2\delta}{2\pi}
\sup_{|\theta|\ge \delta}  h(\rho_1 e^{i\theta}),
\]
thus
\[
\sup_{|\theta|\ge \delta}  h(\rho_1 e^{i\theta}) \ge 
\frac{1-(\delta/2\pi)}{1-(\delta/\pi)} h(0) \ge (1+(\delta/2\pi)) h(0).
\]
Using the fact that $\cos^{-1}(1-u^2)\ge \sqrt2 u$, one can check that $\delta/2\pi > e^{-2r_1} / 10$.
\end{proof*}

We will also need the following technical auxiliary result. 

\begin{lemma}\label{step1}
Let $\Theta$ be an inner function and $0< \eps < 1$, $R_1 >2$ such that the set $\{w \in \D : 0 < |\Theta (w)| < (1-\eps)^{1/4}\}$ contains no hyperbolic disc of hyperbolic radius $R_1$. Then for $R_0$ large enough, depending only on $R_1$, we have 
\begin{equation}\label{conclusion}
|\Theta(z)| \geq 1-\eps  \text{  if  }  d_H(z, \ZTH)\ge R_0.
\end{equation}
\end{lemma}

\begin{proof*}{\it Proof of Lemma \ref{step1}.}

Let $R_0 \ge 4(R_1 +1)$ be a large number to be specified later. We argue by contradiction. Assume that there exists a point $z_0 \in \D$ such that 
$d_H(z_0,\ZTH) \ge R_0$ and $|\Theta(z_0)|\le 1-\eps$. 
Let $H(z):= - \log |\Theta(z)|$, $z \in \D$. We claim that we can construct a finite collection of points $\{z_k\}_{k=0}^m$,
with $m= \lfloor \frac{R_0}{2R_1}-2 \rfloor$,
such that 
\begin{equation}
\label{rechyp}
d_H(z_0,z_k)\le kR_1 \mbox{ and } 
H(z_k) \ge \left(1+\frac1{10} e^{-2R_1}\right)^k \log \frac{1}{1-\eps} .
\end{equation}
We proceed by induction: \eqref{rechyp} is trivially satisfied for 
$k=0$.  Given $k\le m$, the first part of the induction statement implies in
particular that $d_H(z_0,z_k)<R_0/2$,
$D_H(z_k, R_0/2) \subset \{ w \in \D : \Theta (w) \neq 0\}$.  By the choice
of $R_1$, $D_H(z_k,R_1)\not \subset \{ w \in \D: 0<|\Theta(w)|<(1-\eps)^{1/4}\}$, 
but $\Theta$ never vanishes on that disc, so by the minimum principle
there is a point $z'_k$ such that $d_H(z_k,z'_k )=R_1$ 
and $H(z'_k)\le \frac{1}{4} \log \frac{1}{1-\eps} \le \frac{1}{4} H(z_k)$.

So we can apply Lemma \ref{step} to $h:=H \circ \varphi_{z_k}$, $r_1=R_1$, $r_0=R_0/2$. 
We set $z_{k+1}:=
\varphi_{z_k} (\rho_1 e^{i\theta_1})$, so that $d_H(z_k,z_{k+1})=
\tanh^{-1} \rho_1 = R_1$, and
\[
H(z_{k+1})= h(\rho_1 e^{i\theta_1}) \ge \left(1+\frac1{10} e^{-2R_1}\right) h(0)
= \left(1+\frac1{10} e^{-2R_1}\right)H(z_k),
\]
and the inductive step is completed.

Now we claim that $H(z)\ge \log \frac{1}{1-\eps}$ for any $ z \in D_H (z_m, R_1+1)$, which contradicts the assumption on $R_1$.
Indeed, applying Harnack's inequality, we have for $z\in D_H (z_m, R_1+1)$, 
\[\begin{split}
\frac{H(z)}{H(z_m)}  &\ge  \frac{\tanh(2(R_1+1)) - \tanh (R_1+1)}{\tanh(2(R_1+1)) + \tanh (R_1+1)}
\\
&=\frac{e^{-2R_1}}{e^{2}+  e^{-2R_1}+e^{-4R_1-2}}
\ge \frac18 e^{-2R_1}.
\end{split}
 \]

So, for $m$ sufficiently large we have $H(z) \geq - \log (1- \eps )$ for any $z \in D_H (z_m , R_1 + 1)$ as desired.  

 For the record, note that a value of $R_0$ that can be chosen is 
\begin{equation*}
    R_0 = 2R_1 \left( 3+ \frac{2R_1+ \log 8}{\log\left(1+\frac1{10} e^{-2R_1}\right)} \right) . 
 \end{equation*}
\end{proof*}

\begin{proof*}{\it Proof of Theorem \ref{narrowchar}.}

The necessity of narrowness is proved in Proposition \ref{SIP-M}. 

Conversely assume that $\{w: 0<|\Theta(w)|<1-\eps'\}$ is narrow for any $0 < \eps'<1$.
Given $0 < \eps < 1$, we thus have some $R_1\ge 2$ such that  the set $\{w \in \D: 0<|\Theta(w)|<(1-\eps)^{\frac{1}{4}}\}$ contains no hyperbolic disc of hyperbolic radius $R_1$. Then Lemma \ref{step1} provides a constant $R_0 = R_0 (R_1)$ such that $|\Theta (z)| \geq 1- \eps$ if $d_H (z, \Theta^{-1} \{0\} ) \geq R_0$. Hence $\Theta$ has the MAG and by Theorem \ref{motiv} we deduce that $\Theta $ has the SIP. 
\end{proof*}

We now prove Corollary \ref{SIP-equiv}

\begin{proof*}{\it Proof of Corollary \ref{SIP-equiv}.}
Assume that 
there exists $ 0 < t_0 <1$ such that $\eta_\Theta (t_0)>0$. 
Suppose that $\Theta$ does not satisfy the SIP,
by Theorem \ref{narrowchar} there exists $\eps>0$ such that for any $R>0$, there exists
$z_R\in \D$ such that $D_H(z_R,R)$ is contained in $ \{ w \in \D: 0<|\Theta(w)|<1-\eps\}$.  
For any $R'>0$, consider $R=2\max(R', \tanh^{-1}(t_0))$. Then for all $w\in D_H(z_R,R')$, 
\[
d_H(w,\ZTH)\ge d_H(z_R,\ZTH)-d_H(z_R,w)
\ge \frac{R}{2} \ge \tanh^{-1}(t_0),
\]
so that $|\Theta (w)| \ge \eta_\Theta (t_0)$. 
We deduce that 
\begin{align*}
    & D_H(z_R,R')\subset \{ w \in  \D :  \eta_\Theta (t_0) <|\Theta(w)|<1-\eps\}\
    \\
    & \subset \{ w \in \D : \eps_1<|\Theta(w)|<1-\eps_1\},
\end{align*}
where $\eps_1:= \min(\eta_\Theta (t_0),\eps)$. Hence $\Theta\not\in\mathcal M$ and the proof is completed.
\end{proof*}

We now prove Corollary \ref{WEP-cor}

\begin{proof*}{\it Proof of Corollary \ref{WEP-cor}.}
Part (a) follows from the description mentioned before, of the classes $\mathcal{P}$ and $\mathcal{M}$ in terms of the level sets given in \cite{MN} and from Theorem \ref{narrowchar}. 

We now prove part (b). Assume that $\Theta \in \mathcal{M}$ satisfies the WEP. Since $\eta_{\Theta}$ cannot vanish identically, Corollary~\ref{SIP-equiv} gives that $\Theta$ has the SIP. The converse implication follows from part (a). 

To prove part (c), note that by (a), 
any Blaschke product in $\mathcal P $ satisfies the SIP. The converse follows from \cite[Proposition 2]{Bo} which says  that if $0<|a|< \sup \{ \eta_\Theta (t) : 0<t<1 \}$, then $\varphi_a \circ \Theta$
is a Carleson--Newman Blaschke product.


\end{proof*}

\section{Proof of Theorem \ref{NSC-P}}
\label{classP}

The following auxiliary result can be found in \cite[Theorem 2.2]{MN}, see also \cite{GM}. 
\begin{lemma}
\label{CN-narrow1}
 Let $\Theta$ be a Blaschke product with zeros $\{a_j\}$. The following conditions are equivalent:
\begin{enumerate}
 \item $\Theta$ is a Carleson--Newman Blaschke product.
 \item There exist $0 < r < 1$ and $ 0 < \eps  < 1$ such that $ D_\rho(a_k,r)\not\subset\{w \in \D: |\Theta(w)|<\eps\}$, for any $k \geq 1$.
 \end{enumerate}
\end{lemma}

 Recall that given $t\in (0,1)$ and a discrete sequence $\{a_j\}$ in $\D$, we denote 
\[
S_{t}(z)=\sum_{j:\rho(z,a_j)\ge t} 1-\rho(z,a_j).
\]
The proof of Theorem \ref{NSC-P} is organized in several steps. 

\begin{proposition}\label{sufficient}${}$ Let $\Theta$ be a Blaschke product with zeros $\{a_j\}$. Then $\Theta \in \mathcal{P}$ if and only if    
  \begin{equation}\label{CS}
\lim_{t\to 1}\sup_{z \in \D} S_t (z)=0 . 
\end{equation}
\end{proposition}
\begin{proof} 
We start proving the necessity of condition \eqref{CS}.  Fix $\eps >0$. Since $\Theta \in \mathcal{P}$, there exists $R_1 >2$ such that the set $\{w \in \D: |\Theta(w)|<(1-\eps)^{\frac{1}{4}}\}$ contains no hyperbolic disc of hyperbolic radius $R_1$. Let $R_0$ be obtained from $R_1$ by Lemma~\ref{step1}. Fix $z \in \D$ and denote by $\tilde \Theta$ the Blaschke product whose zeros are the zeros of $\Theta$ which are at hyperbolic distance from $z$ bigger than $R_0$, that is,
$$
\tilde \Theta =\prod_{j: d_H(z_0,a_j)\ge R_0}b_{a_j} , \qquad b_{a} (z) = \frac{|a|}{a} \frac{a-z}{1- \overline{a} z}, \qquad z \in \D. 
$$
Since $|\Theta| \leq |\tilde \Theta|$ on $\D$, the set $\{w \in \D: |\tilde \Theta(w)|<(1-\eps)^{\frac{1}{4}}\}$ contains no hyperbolic disc of hyperbolic radius $R_1$. 
Then Lemma~\ref{step1} gives that $|\tilde \Theta (z)| \ge 1 - \eps$. We deduce that there exists $0 < t_0 = t_0 (R_0) < 1$, $t_0 \to 1$ as $ R_0 \to  \infty$, such that
$$
S_t (z) \lesssim - \log |\tilde \Theta (z)|  \leq  - \log (1- \eps), \quad t_0 < t < 1. 
$$
This is \eqref{CS}. 

 


We now prove the converse. If $\rho(z, \ZTH) \ge 1/2 $, then $- \log |\Theta (z)| \le 2 S_t (z)$.
Therefore, $\eta_\Theta(t)\ge e^{-2\sup_{\D} S_t}$ for $1/2 < t <1$. Thus (\ref{CS}) readily implies that $\Theta$ verifies the SIP. 

Furthermore, there exists $t_0\in (0,1)$ such that $\sup_{\D} S_{t_0}\le 1$. Let   $z,w\in \D$ be such that $\rho(z,w)=
2t_0 / (1+t_0^2)$. Using \cite[Lemma 1.4]{Ga}, we have 
\[
\begin{split}
\sum_{j:\rho(z,a_j)<t_0}(1-\rho(z,a_j))&\le \sum_{j:\rho(w,a_j)\ge t_0}(1-\rho(z,a_j))\\
&\le \frac{(1+t_0)^2}{(1-t_0)^2} \sum_{j:\rho(w,a_j)\ge t_0}(1-\rho(w,a_j)).
\end{split}
\]
Finally, for all $z\in \D$,
\[
\sum_{j}(1-\rho(z,a_j))\le 1+\frac{(1+t_0)^2}{(1-t_0)^2},
\]
which implies that $\Theta$ is a Carleson--Newman Blaschke product. Then Corollary~\ref{WEP-cor} completes the proof.
\end{proof}

%

Let us recall that for $\theta\in \R$ and $h, \delta >0$, we denote   
\[Q(\theta,h,\delta):=\{re^{i\theta'}: 0<1-r< \delta h,\ |\theta-\theta'|< h\}.
\]

\begin{proposition}${}$ Let $\Theta$ be a Blaschke product with zeros $\{a_k \}$. Let 
\[
\mu = \sum_j (1- |a_j|) \delta_{a_j} . 
\]
Then \eqref{CS} holds if and only if
\begin{equation}\label{box}
\lim_{\delta\to 0}\sup_{\theta\in \R, h>0} \frac{\mu(Q(\theta,h,\delta))}{h}=0
\end{equation}

\end{proposition}
\begin{proof}${}$
 \eqref{CS} $\Longrightarrow$ \eqref{box}: Let $\theta\in \R$, $h\in(0,1)$. We set $z=(1-h)e^{i\theta}$. Note that for any $0< \delta < 1/2$ there exists $0 < t(\delta ) < 1$, $ t(\delta ) \to 1$ as $\delta \to 0$ such that $\rho (z, a_j ) \geq t(\delta)$ for any $a_j \in Q(\theta, h ,  \delta)$. Since there exists a universal constant $C>0$ such that  $1-|a| \leq C (1- \rho(a,z))h$ for any $a\in Q(\theta, h, \delta)$, we deduce that $\mu(Q(\theta,h,\delta)) \leq C S_{t(\delta)} (z) h$ and \eqref{box} follows.
 
\eqref{box} $\Longrightarrow$ \eqref{CS}: First, observe that \eqref{box} implies the existence of a fixed  $0 < \delta_0 < 1$ such that $\mu(Q(\theta,h,\delta_0)) \leq h$ for any $\theta \in \R , h>0$. Then 
\begin{equation}
\label{fixed}
\mu(Q(\theta,h,1))\le \mu\left(Q(\theta,h \delta_0^{-1},\delta _0)\right)\le h\delta_0^{-1}, 
\end{equation}
for any $\theta\in \R$ and $h>0$. In other words,  \eqref{box} implies that $\mu$ is a Carleson measure and hence $\Theta$ is a Carleson--Newman Blaschke product.

Next, we fix  $z=(1-h_0)e^{i\theta}\in \D$ and use the partition 
\[\D=\bigcup_{n\in \N} (Q_n\setminus Q_{n-1})\] where we have denoted
\[Q_{-1}=\emptyset, \ \ \ \ Q_n=Q(\theta,2^nh_0,1)\ \ \   n\in \N.\]
There exist two constants $\alpha,\beta>0$ such that 
$$
\alpha 2^{-2n}  \frac{1-|a_j|}{h_0} \le  1-\rho(z,a_j)\le \beta 2^{-2n}  \frac{1-|a_j|}{h_0},\quad n\in \N,\, a_j\in Q_n\setminus Q_{n-1}.
$$
Let $m\in \N$. If $1-\rho(z,a_j)\le \alpha 2^{-2m}$,  $a_j\in Q_n\setminus Q_{n-1}$, and $n\le m$, then $1-|a_j|\le  2^{2(n-m)}h_0$. We obtain
\[
\begin{split}
&S_{1-2^{-2m}}(z)
\\
&\le \beta \sum_{n=0}^m 2^{-2n}  \sum_{a_j\in  Q(\theta,2^nh_0,2^{-m})}\frac{1-|a_j|}{h_0}+ \beta \sum_{n=m+1}^{+\infty}2^{-2n}  \sum_{a_j\in  Q_n\setminus Q_{n-1}}  \frac{1-|a_j|}{h_0}\\
&\le \beta \sum_{n=0}^m 2^{-n}\frac{\mu(Q(\theta,2^nh_0,2^{-m}))}{2^nh_0}+ \beta \sum_{n=m+1}^{+\infty} 2^{-n}\frac{\mu(Q(\theta,2^nh_0,1))}{2^nh_0}\\
&\le 2\beta \sup_{\theta\in \R, h>0} \frac{\mu(Q(\theta,h,2^{-m}))}{h}+ \beta 2^{-m}\sup_{\theta\in \R,h>0} \frac{\mu(Q(\theta,h,1))}{h}.
\end{split}
\]
By (\ref{box}) and (\ref{fixed}),  we obtain
\[\lim_{m\to \infty}\sup_{z\in \D}S_{1-2^{-2m}}(z)=0. \]
\end{proof}

Next we will prove the equivalence between conditions (b) and (d) in Theorem~\ref{NSC-P}, that is, we prove that the set $\mathcal P$ is precisely the set of inner functions such that all their divisors satisfy the SIP. It will be convenient to work with the hyperbolic distance and consider the equivalent quantity
\[
\tilde S_x(z) := \sum_{a\in \ZTH\setminus D_H(z,x)}
e^{-2d_H(a,z)}, \qquad x>0,\, z \in \D.
\]
We start with an auxiliary result
\begin{lemma}
\label{bigfar}
Assume that $\sup_{z \in \D} \tilde S_x (z)> \delta$ for any $x\in (0,\infty)$. Then for any $z\in \D$ and any $A>0$, there exists $x=x(z,A)>A$
and $w\in \D$ such that $d_H(z,w)>A$ and $\tilde S_x(w)>\delta$.
\end{lemma}
\begin{proof}
By the assumption on $\tilde S$, it will be enough to
show that $\tilde S_x(w)\le\delta$ if $d_H(w,z)\le A$ and $x$ is sufficiently large.
Take $x$ large enough so that $\tilde S_{x-A} (z) < e^{-2A} \delta$.
Then by using the triangle inequality twice we obtain that 
\begin{multline*}
\tilde S_x (w) = \sum_{a\in \ZTH\setminus D_H(z,x)}
e^{-2d_H(a,w)}
\\
 \le 
\sum_{a\in \ZTH\setminus D_H(z,x-A)}
e^{-2 d_H(a,z)+2A} = e^{2A} \tilde S_{x-A} (z) 
<\delta.
\end{multline*}
\end{proof}

Now we pass to the equivalence between conditions (b) and (d) in Theorem~\ref{NSC-P}. 

\begin{proposition}
Let $\Theta$ be an inner function with zeros $\{a_j\}$. 
Then $\Theta$ is a Blaschke product satisfying 
the uniform Blaschke tails condition \eqref{CS}
if and only if every divisor of $\Theta$ has the {\rm SIP}.
\end{proposition}

\begin{proof}
We have already seen in Proposition \ref{sufficient} that \eqref{CS} is sufficient for 
a Blaschke product to have the SIP, and since condition \eqref{CS} is preserved under removal of any subset of the zeros, it 
will also imply that Blaschke products which divide $\Theta$
have the SIP.

Conversely, if $\Theta$ has a non-trivial singular factor, 
then it will be a divisor which fails the SIP, so we can assume that $\Theta$ is a Blaschke product. We proceed by contradiction, and assume that there exists
$\delta>0$ such that $\sup_{z \in \D} \tilde S_x (z)> \delta$ for any $x\in (0,\infty)$.

Using Lemma~\ref{bigfar} we now construct in an inductive argument a sequence of disjoint hyperbolic discs 
$D_H(w_k,k)$. A divisor $\tilde \Theta$ failing the SIP 
will be obtained
by removing the zeros of $\Theta$ lying in those discs,
so that $\mathcal Z(\tilde \Theta) = \ZTH\setminus (\cup_{k\ge 1}D_H(w_k,k))$.

We start by choosing $w_1, x_1$ with $\tilde S_{x_1}(w_1)>\delta$.

Then we choose $R_1 > x_1$ such that 
$$
\sum_{a\in\ZTH\cap D_H(w_1,R_1)\setminus D_H(w_1,x_1)}e^{-2d_H(a,w_1)} \ge \frac34 \delta 
$$
and
$$
\sum_{a\in\ZTH\cap D_H(w_1,1)} e^{-2d_H(a,w)} \le \frac14 \delta,\qquad  d_H(w,w_1)>R_1.
$$

On step $k\ge 2$ we apply Lemma~\ref{bigfar} with $z=w_1$ and 
$$
A=\max_{1\le s<k}d_H(w_1,w_s)+R_{k-1}+k
$$ 
to get $w_k=w$ and $x_k=x$ such that 
$d_H(w_1,w_k)> A$, $x_k>A$, and $\tilde S_{x_k}(w_k)>\delta$.
Next we choose $R_k> x_k$ such that 
$$
\sum_{a\in\ZTH\cap D_H(w_k,R_k)\setminus D_H(w_k,x_k)  }e^{-2d_H(a,w_k)} \ge \frac34 \delta 
$$
and
$$
\sum_{a\in\ZTH\cap (\cup_{1\le s\le k}D_H(w_s,s))} e^{-2d_H(a,w)} \le \frac14 \delta,\qquad  d_H(w,w_1)>R_k.
$$
This completes the induction step. 

Then $D_H(w_k,R_k)\cap D_H(w_j,j)=\emptyset$, $j>k$ and $k< x_k$. Therefore,  
$$
\sum_{a\in\ZTH\setminus (\cup_{j\ge k} D_H(w_j,j)) } e^{-2d_H(a,w_k)} \ge \frac34 \delta,\qquad  k\ge 1.
$$
Next, $d_H(w_k,w_1)>R_{k-1}$, and hence,
$$
\sum_{a\in\ZTH\cap (\cup_{1\le j\le k-1} D_H(w_j,j))} e^{-2d_H(a,w_k)} \le \frac14 \delta,\qquad  k\ge 1.
$$



The points $w_k$ verify $d_H(w_k, \mathcal Z(\tilde \Theta)) \ge k$, and
\begin{multline*}
\log \frac{1}{|\tilde \Theta(w_k)|} \gtrsim
\sum_{ a\in \mathcal Z(\tilde \Theta)} e^{-2d_H(a,w_k)}
\\
=  \sum_{a\in\ZTH\setminus (\cup_{j\ge k} D_H(w_j,j))} e^{-2d_H(a,w_k)}   -
\!\!\!  \sum_{a\in\ZTH\cap (\cup_{1\le j\le k-1} D_H(w_j,j))} e^{-2d_H(a,w_k)}
\\
\ge 
  \frac{\delta}{2},\qquad  k\ge 1.
\end{multline*}
This shows that $\tilde \Theta$ does not have the SIP. 
\end{proof}



\section{Proof of Theorem \ref{sipentropy}}
\label{divsip}
We will use the family of dyadic arcs of the unit circle given by 
$$
\{e^{i \theta} : k 2^{-m} \leq \theta / 2\pi < (k+1) 2^{-m} \}, \qquad 0\le k<2^m,\, m\ge 1.
$$
When $I$ is an arc on the unit circle we denote by $2I$
the arc with the same center and twice the length.



\begin{proof*}{\it Proof of Theorem \ref{sipentropy}.}

\ {}
{\it Proof of ${\rm(a)} \Rightarrow {\rm(b)}$.}
We follow the notations and several arguments from \cite{BNT}.
Let $\mathcal G=\{J_n\}$ be the family of maximal dyadic arcs of $\partial \D$ 
such that $2J_n \subset \partial \D \setminus E$. The arcs $\{J_n \}$ form a sort of Whitney
decomposition of $\partial \D \setminus E$. From \cite[Lemma 12(b)]{BNT},
we know that if $E$ has finite entropy, then 
\begin{equation}
\label{entropint}
\sum_{J_n \in \mathcal G} |J_n| \log\frac{1}{|J_n|} < \infty.
\end{equation}
Let $\mathcal F$ be the family of dyadic arcs $I$ of
$\partial \D$ which are not contained in any arc of ${\mathcal G}$, that is,  $2I \cap E\neq \emptyset$. 
As in \cite[pp. 1000--1001]{BNT}, we see that 
\[
\sum_{I \in \mathcal F} |I|  < \infty. 
\]
Let $B_1$ be the Blaschke product with zeros $\{z_I :  I \in \mathcal F\}$,
and $f:=B_1 S_\mu$. 

As before, for an arc $J$ of $\partial \D$, let $Q(J)= \{ z \in \D: \frac{z}{|z|} \in J, 1-|z|<|J|\}$ be the Carleson square based at $J$.

{\bf Claim. }  
For any $J\in \mathcal G$, $z\in Q(J)$, we have 
\[
\log \frac{1}{|f(z)|} \lesssim \frac{1-|z|}{1-|z_J|} \log \frac{1}{|f(z_J)|}
\lesssim \frac{1-|z|}{|J|^2}.
\]
{\it Proof of the Claim.}
Since $z\in Q(J)$ and $J\in \mathcal G$, we have $\rho(z_J , \mathcal Z(B_1)) 
\gtrsim 1$. Then
\[
\log \frac{1}{|f(z)|} \lesssim P[\mu] (z) 
+ \sum_{I \in \mathcal F} \frac{(1-|z_I|)(1-|z|)}{|1-\overline{z_I} z|^2}.
\]
Since $J\in \mathcal G$, $2J\cap \mbox{supp }\mu = \emptyset$ and $Q(2J) \cap \mathcal Z(B_1) = \emptyset$. So for any $z\in Q(J)$ and $\xi \in \mbox{supp }\mu $, $|\xi-z| \simeq |\xi-z_J|$.
If furthermore $I \in \mathcal F$, then $|1-\overline{z_I} z| \simeq |1-\overline{z_I} z_J|$.
Thus
\[
P[\mu] (z) \simeq \frac{1-|z|}{1-|z_J|} P[\mu] (z_J) 
\mbox{ \ and \ }
\frac1{|1-\overline{z_I} z|^2} \simeq \frac1{|1-\overline{z_I} z_J|^2} .
\]
We deduce that 
\begin{multline*}
\log \frac{1}{|f(z)|} \simeq 
 \frac{1-|z|}{1-|z_J|}
 \left(
 P[\mu] (z_J) + \sum_{I \in \mathcal F} \frac{(1-|z_I|)(1-|z_J|)}{|1-\overline{z_I} z_J|^2}
 \right)
 \\
 \\ \lesssim
 \frac{1-|z|}{1-|z_J|} \log \frac{1}{|f(z_J)|},
\end{multline*}
which yields the first estimate. For the second one, observe that by Harnack's inequality,
$P[\mu] (z_J)\lesssim |J|^{-1}$ and 
\[
\sum_{I \in \mathcal F} \frac{(1-|z_I|)(1-|z_J|)}{|1-\overline{z_I} z_J|^2}
\lesssim
\frac{1}{1-|z_J|}\sum_{I \in \mathcal F} 1-|z_I| 
\lesssim 
\frac1{|J|}.
\]
The Claim is proved.  \hfill \qed

Next, we 
consider 
\[
\mathcal L := \bigcup_{J\in \mathcal G} 
\{ L \mbox{ dyadic arc: } L\subset J, |L| \ge |J|^2\}.
\]
Note that
\[
\sum_{J \in \mathcal L}|J|  
\simeq 
\sum_{J \in \mathcal G} |J| \log_2 |J|^{-1} < \infty,
\]
by \eqref{entropint}. So the sequence $\{ z_L :  L \in \mathcal L\}$
satisfies the Blaschke condition. Let $B_2$ be the corresponding 
Blaschke product. We will prove that $fB_2 = S_\mu B_1 B_2$ has the SIP. We need to show that $|f (z) B_2(z)|$ is close to $1$ when $\rho(z, \mathcal Z(B_1B_2))$ is sufficiently close to $1$.

The Claim gives that $\log |f(z)|^{-1} \le \eps$ whenever $z\in Q(J)$ and 
$1-|z| \le \eps |J|^2$, for some $J \in \mathcal G$. 
Given $\eps>0$, there exists $\delta>0$ such that
\[
\left\{ 
z: \rho(z, \mathcal Z(B_1B_2)) \ge 1-\delta
\right\}
\subset
\bigcup_{J \in \mathcal G} 
\left\{ 
z\in Q_J: 1-|z| \le \eps |J|^2
\right\}.
\]
Then it only remains to show that $\log |B_2(z)|^{-1}\to0$ as $\eps \to 0$
when $z\in Q(J)$, $1-|z| \le \eps |J|^2$, for some $J \in \mathcal G$.
For such a point $z$, let $j_0:=j_0(z)$ be the smallest positive integer 
such that $2^{j_0} Q(z)$ contains a zero of $B_2$, where $Q(z)=Q(L_z)$ with $L_z$
the shortest dyadic arc $L$ such that $z \in Q(L)$. 
We have $j_0(z) \gtrsim \log_2 \eps^{-1}$.
Then
\begin{multline*}
\log |B_2(z)|^{-1} 
\simeq \sum_{L \in \mathcal L} 
 \frac{(1-|z_L|)(1-|z|)}{|1-\overline{z_L} z|^2}
 \\
 \lesssim 
 \sum_{j\ge j_0}  \frac1{2^{2j} (1-|z|)} 
 \sum_{L\in \mathcal L,\, z_L \in Q(2^jL_z)\setminus Q(2^{j-1}L_z)} 1-|z_L| 
 \\ \lesssim 
 \sum_{j\ge j_0} \frac{j2^j}{2^{2j} } \lesssim  \frac{j_0}{2^{j_0} } 
 \lesssim \eps \log \eps^{-1}.
\end{multline*}
This completes the proof of the implication ${\rm(a)} \Rightarrow {\rm(b)}$.

\vskip.3cm
{\it Proof of ${\rm(b)} \Rightarrow {\rm(a)}$.}
Lemma~4 of \cite{BNT} establishes that a closed set $E \subset \partial \D$ has finite entropy if and only if for any positive Borel singular measure $\mu$ supported on $E$ and any constant $C>0$, one has 
$$
\sum |J| < \infty ,
$$
where the sum is taken over all dyadic arcs $J \subset \partial \D$ satisfying the property $P[\mu] (z_J) \ge C$. 
Arguing by contradiction, suppose that there exists a singular measure $\mu$
supported on $E$ 
such that 
\[
\sum |I| = \infty,
\]
where the sum is taken over all dyadic arcs $I$ such that $P[\mu](z_I) \ge 2$. By Harnack's inequality, we have $P[\mu](z) \ge 1$ when $z\in D_\rho (z_I, 1/3)$
and we deduce
\[
\int_{\{z\in \D: P[\mu](z) \ge 1 \}} \frac{dA(z)}{1-|z|} = \infty,
\]
where $dA$ denotes the area measure. This will contradict the next Lemma,
which parallels the analogous fact for inner functions having the WEP, \cite[Proposition 4]{Bo}.

\begin{lemma}
If $S_\mu$ is a divisor of an inner function having the {\rm SIP}, then for any $C>0$,
\[
\int_{ \{z\in \D: P[\mu](z) \ge C \}} \frac{dA(z)}{1-|z|} < \infty . 
\]
\end{lemma}

\begin{proof}
Let $B$ be an inner function such that $\Theta=B S_\mu$ has the SIP. Using that $|\Theta(z)| \le |S_\mu(z)|$, $z \in \D$ and that $\Theta$ has the SIP, for any $C>0$ we have 
\begin{multline*}
\left\{ z\in \D: P[\mu](z) \ge C \right\}
= \left\{ z\in \D: |S_\mu(z)| \le e^{-C} \right\}
\\
\subset
\left\{ z\in \D: |\Theta(z)| \le e^{-C} \right\}
\\\subset
\left\{ z\in \D: \rho(z, \mathcal Z(B)) \le \varkappa_\Theta(e^{-C}) \right\}=: \Gamma.
\end{multline*}
The Blaschke condition on the zeros of $B$ gives that 
$$
\int_\Gamma \frac{dA(z)}{1-|z|} < \infty , 
$$
which finishes the proof. 
\end{proof}
\end{proof*}

\section{Examples and further results}
\label{examples}
\subsection{Thin Blaschke products, super-separated zeros and the SIP}
\label{thinbp}
A sequence $\{a_k\}$ of points in the unit disc is called an exponential sequence if 
$$
\limsup_{k \to \infty} \frac{1- |a_{k+1}|}{1-|a_k|} < 1.
$$
\begin{proposition}\label{Corol1}${}$

 {\rm (a)} Any finite product of thin Blaschke products  satisfies the SIP.

 {\rm (b)} Any Blaschke product whose zero set is a finite union of exponential sequences satisfies the SIP. 
\end{proposition}

Note that Blaschke products in both of these classes are Carleson--Newman, therefore by part (c) of Corollary \ref{WEP-cor} belong to $\mathcal P$. 

\begin{proof}
In part (a) one only needs to show that a thin Blaschke product satisfies the MAG condition, which by Theorem \ref{motiv} means that it satisfies the SIP. This follows from classical estimates of K.Hoffmann, see \cite[Lemma 1.4 of Chapter X]{Ga}.

For (b), Theorem~\ref{NSC-P}~(c) shows that the Blaschke product is in the class $\mathcal{P}$ and hence satisfies the SIP. 
\\
\end{proof}

Of course, many SIP Blaschke products are not thin (for instance the Blaschke product with zeros $ 1- 2^{-j}$, $j \geq 1$). 
However, if the zeros of an inner function are sufficiently separated, then the SIP implies thinness. 
Recall that a sequence $\{a_k \} \subset \D$ is called \emph{super-separated} if for any $\delta>0$, there exists
$N\in \N$ such that $\rho(a_j,a_k) \ge 1-\delta$ for any $j>k\ge N$. 

\begin{proposition}
\label{supersep}
Let $\Theta$ be an inner function whose zeros are super-separated. Assume that $\Theta$ satisfies the {\rm SIP}. Then $\Theta$ is a thin Blaschke product. 
\end{proposition}

\begin{proof}
The assumptions imply that for any $\theta\in \R$ one can find a sequence $r_n \to 1$
such that $|\Theta(r_ne^{i\theta})| \to 1$. Hence $\Theta$ cannot have a non-trivial singular inner factor.
Now we need to show that $(1-|a_k|) \Theta'(a_k)|$ can be made arbitrarily close to $1$ when $k$
is large enough. Here $\{a_k \}$ denotes the sequence of zeros of $\Theta$. Fix $0 < \eps < 1$. By the SIP, there exists $0 < \delta < 1$ such that 
$|\Theta(z)|\ge 1-\eps$ whenever $\rho(z, \{a_k \} ) \ge 1-\delta$. Pick an integer $N=N(\delta) >0$ large enough so that   
for $k\ge N$, we have that $D_\rho (a_k, 1-\delta)$
is at $\rho$-distance at least $1-\delta$ from any other zero of $\Theta$, so that $|\Theta(z)/\varphi_{a_k}(z)| \ge |\Theta(z)| \ge 1-\eps$ for any $z \in \partial D_\rho (a_k, 1-\delta)$. Since $\Theta(z)/\varphi_{a_k}$ is zero-free on the disc $D_{\rho} (a_k , 1 - \delta)$, we obtain that $(1-|a_k|^2) \Theta'(a_k)|\ge 1-\eps$
by applying the minimum modulus principle to the function $\Theta / \varphi_{a_k}$ at $z=a_k$.
\end{proof}

We will see that a divisor of a SIP Blaschke product is not, in general, SIP; however Blaschke products with a super-separated zero sequence
(and so in particular thin Blaschke products) are in a sense negligible elements for the SIP. In other words an inner function can be multiplied
or divided by such a Blaschke product without changing its status with respect to the SIP.

\begin{proposition}
\label{unionsupersep}
Let $\Theta_1$ and $\Theta_2$ be two inner functions. Assume that $\Theta=\Theta_1 \Theta_2$  satisfies the SIP and that $\mathcal Z(\Theta_1)$ is finite or super-separated, then $\Theta_2$ satisfies the SIP.
\end{proposition}

\begin{proof}
 Assume first that $\mathcal Z(\Theta_1)$ is finite. By induction, we may assume that it only contains a single point $a_1$. It suffices to show that $\eta_{\Theta_2}(\tanh (3t))\ge \eta_{\Theta}(\tanh t)$ for $t>0$. Let $z\in \D$ be such that 
 $$
 d_H(z,\mathcal Z(\Theta_2))\ge 3t. 
 $$
We consider two cases. 
 
If $d_H(z,a_1)\ge t$, then $|\Theta_2(z)|\ge |\Theta(z)|\ge \eta_\Theta(\tanh t)$. 

If $d_H(z,a_1)< t$, then all $w \in \D$ such $d_H(z,w)=2t$ satisfy the relation $d_H(w,\ZTH)\ge t$, thus $|\Theta_2(w)|\ge |\Theta(w)|\ge \eta_\Theta(\tanh t)$ and since $\Theta_2$ does not have zeros  on  $D_H(z,2t)$ we also have $|\Theta_2(z)|\ge \eta_\Theta(\tanh t)$.  
  
Now we assume that $\mathcal Z(\Theta_1)=\{a_j\}$  is super-separated.
 
Let  $t>0$ and $z\in \D$ be such that $d_H(z,\mathcal Z(\Theta_2))\ge 3t$. There exists $N\in \N$ such that $d_H(a_k, a_j)\ge 4t$ for all $j>k> N$. Applying  the previous case, we see that $\tilde \Theta:=\Theta \prod_{j\le N} b_j^{-1}$, where 
$$
b_j(z):= \frac{\bar a_j}{|a_j|} \frac{a_j-z}{1-\bar a_j z}, \quad z \in \D , 
$$ 
satisfies the SIP.

If  $d_H(z,a_j)\ge t$ for all $j>N$, then $|\Theta_2(z)|\ge|\tilde \Theta(z)|\ge \eta_{\tilde \Theta}(\tanh t)$.\\
Assume that  $d_H(z,a_j)< t$ for some $j>N$, then  all $w$ such that $d_H(z,w)= 2t$ satisfy  $d_H(w, \mathcal Z(\tilde \Theta)) \ge t$. We conclude like in the previous case that 
$\eta_{\Theta_2}(\tanh (3t))\ge \eta_{\tilde\Theta}(\tanh t)$.
\end{proof}

\begin{corollary}
Let $\Theta_j$, $1\le j\le N$, be inner functions, each one with super-separated zeros and let the product  $\prod_{1\le j\le N}\Theta_j$ satisfy the {\rm SIP}. Then each $\Theta_j$ is a thin Blaschke product.
\end{corollary}

\begin{proof}
By induction  and  Proposition  \ref{unionsupersep}, we show that each $\Theta_j$ satisfies the SIP. Then, we apply Proposition \ref{supersep}  to deduce that each $\Theta_j$ is a thin Blaschke product. 
\end{proof}

\subsection{Carleson--Newman Blaschke products, the WEP, and the SIP}

\begin{proposition}
\label{wepsipnotCN}
There exist Blaschke products satisfying both the SIP and the WEP which are not Carleson--Newman.
\end{proposition}

\begin{proof}
This construction is due to S.Treil as published in \cite{GMN}.  Once again, we carry out 
the computations in the upper half plane $\mathbb H$.  Recall that for $z,w \in \mathbb H$, 
\[
1-\rho(z,w)^2 = \frac{4 \Im z \Im w}{(\Re z-\Re w)^2 + (\Im z+ \Im w)^2}.
\]

Set $y_n=n^3$, $n\ge 1$, and choose 
a decreasing sequence $\{\delta_n\}$, with $\delta_1=1$ and $\delta_n\to 0$, 
such that $\sum_n (\delta_n y_n)^{-1}<\infty$. This ensures that $z_{n,k}:= k\delta_n y_n + i y_n$, $n \in \N$, $k\in \Z$, defines a Blaschke sequence in $\mathbb H$.
Let $B$ denote the associated Blaschke product. It is proved in \cite{GMN} that $B$ satisfies the WEP, and it is obvious that $B$ is not Carleson--Newman.

It is easy to see that there exists $\eps_0>0$ such that 
$d_H (z, \mathcal Z(B))\ge 1-\eps_0$ implies that $\Im z \le \frac14$. Fix $0 < \eps < \eps_0$. Suppose that $z=x+iy$,
$0<y\le \frac14$, and that $\min_{n,k} \rho(z,z_{n,k}) \ge 1-\eps$. If we choose $k_0\in \Z$ such that $|x-k_0|\le \frac12$, we have  
\[
2\eps \ge 1-\rho(z,z_{0,k_0})^2 \ge \frac{4y}{\frac14 + (y+1)^2} \ge 2y,
\]
so $y\le \eps$. Using Riemann sums, we then can estimate
\begin{multline}
\log |B(z)|^{-1} \lesssim \sum_{n\ge 0, k\in \Z} 1-\rho(z,z_{n,k})^2
\\
\lesssim 
\sum_{n\ge 0} 4 y y_n \sum_{k\in \Z} \frac1{(x- k \delta_n y_n)^2 +y_n^2}
\lesssim
y \sum_n \frac1{\delta_n y_n} \lesssim \eps,
\label{estdo4}
\end{multline}
so $B$ satisfies the SIP.
\end{proof}

\begin{corollary} There exist a Blaschke product satisfying the {\rm SIP} and not satisfying the {\rm WEP}.
\end{corollary}

\begin{proof} We just take the Blaschke product $B$ from Proposition~\ref{wepsipnotCN} and remove its zeros in the hyperbolic discs $D_H(8^ni,1)$, $n\ge 1$. 
Then the remaining Blaschke product $B_1$ does not satisfy the WEP because the discs $D_H(8^ni, 2)$ contain an increasing number of zeroes, while their centers $8^ni$ are at a fixed distance from the zero set.  Next, $B_1$ satisfies the SIP because all points which are at 
pseudohyperbolic distance more than 
$1-\eps_0'$ of the zero set, for some $\eps_0'<
\min(\eps_0,1-\tanh 2)$ 
must verify $\Im z \le \frac14$, and 
we just use estimate \eqref{estdo4} and the fact that $|B_1|\ge |B|$.
\end{proof}

Next we provide examples of Interpolating Blaschke products which do  not satisfy the SIP. Our examples are very close to those in the proof of Theorem 5 in \cite{Bo}. 

\begin{proposition}
\label{notstar}
There exist interpolating Blaschke products (which satisfy therefore the {\rm WEP}) which 
do not satisfy the {\rm SIP}.
\end{proposition}

\begin{proof}
For notational convenience, we construct a family of such examples in the upper half plane. For any
interval $I\subset \R$, let $\tilde{I}$ be the interval with the same center and triple the length.

Let $\{L_n \}$
be a decreasing summable sequence of positive numbers, and $I_n := [x_n, x_n +L_n] \subset \R$ a family of
 intervals such that the intervals $\{ \tilde I_n \}$ are pairwise disjoint, $x_n+2L_n \le x_{n+1}$ for any $n$ and all contained in a fixed bounded
 interval.  
 Pick an increasing sequence of integers $\{N_n\} \to \infty$, and define
\[
z_{n,k}:= x_n + \frac{k L_n}{N_n} + i  \frac{ L_n}{N_n} , \quad 0\le k < N_n.
\]
Let $B$ be the Blaschke product with zeros $\{ z_{n,k}:  0\le k < N_n ,  n\ge 1 \}$.
The Blaschke condition is verified because of the summability assumption of $\{L_n \}$. It is clear that $B$ is an interpolating Blaschke product.

Note that there exists a universal constant $0< c_0 < 1$ such that $|B(z)| \leq c_0$ for any $z \in  Q_n:= \{x+iy: x\in I_n,  L_n / N_n \le y \le  L_n \}$. Pick $w_n \in Q_n$ with $\Im (w_n) = L_n$. Note that $\rho (w_n , \mathcal Z(B)) \to 1$ as $n \to \infty$.  
Hence $B$ does not satisfy the MAG condition.



\end{proof}

Last example can be improved to obtain a Blaschke product which cannot even become SIP after the addition
of extra zeros.

\begin{proposition}
\label{intnotsipable}
There exists an interpolating Blaschke product which is not a divisor of a $\mathrm{SIP}$ function.
\end{proposition}

\begin{proof}
Consider the Blaschke product $B$ of Proposition~\ref{notstar} above. 
This time choose $L_n, N_n$ such that $\sum_n L_n <\infty$, $\sum_n L_n \log N_n =\infty$. As before, $B$ is an interpolating Blaschke product.

We argue by contradiction. Assume that $B_1$ is a Blaschke product such that $BB_1$ has the SIP. Since $|BB_1(z)|\le |B(z)| \le c_0$ for any $z \in Q_n$, there exists a zero $\zeta$ of $BB_1$ such that $\rho(z,\zeta) \le \eta_{BB_1}^{-1} (c_0)$.
The Blaschke sum of $BB_1$ on $Q_n$ is thus bigger than a fixed multiple of 
\[
\int_{Q_n}\frac{dm_2(z)}{1-|z|}\asymp L_n \log N_n.
\]
So the zeros of $BB_1$ cannot satisfy the Blaschke condition. 
\end{proof}

\subsection{Divisors of SIP are not always SIP}

Our next result says that divisors of SIP Blaschke products may not be SIP.

\begin{proposition}
 \label{divnotsip}
There exists a Blaschke product $B = B_1 B_2$ which has the {\rm SIP} such that $B_i$ do not have {\rm SIP} or {\rm WEP}, $i=1,2$.   
\end{proposition}

\begin{proof}
    We argue in the upper half plane. Let $B$ be the Blaschke product constructed by S.Treil described in Proposition \ref{wepsipnotCN}. Let $B_1$ (respectively $B_2$) be the Blaschke product having the zeros of $B$ with positive (respectively negative) real part. Note that for any $c>0$ we have 
    $$
   \lim_{x \to + \infty} B_2 (x + i cx) =0, \quad   \lim_{x \to + \infty} B_1 (-x + i cx) =0.    $$
    Hence $B_i$, $i=1,2,$ does not have the SIP or the WEP. 
\end{proof}


\subsection{Stability of the zero set under local perturbations}

\begin{proposition}
\label{jiggle}
 Let $\Theta$ be an inner function, $\{z_n\}_{n\ge 1}\subset \ZTH$, and let $\{z_n^*\}_{n\ge 1}$ be such that $\sup_{n\ge 1}\rho(z_n,z^*_n) <1$. Denote by $B$ (respectively $B^*$) the Blaschke products
with zero sets $\{z_n\}$ (respectively $\{z^*_n\}$). 
If $\Theta $ has the {\rm SIP}, then $\Theta_1=\Theta B^* / B$ has the {\rm SIP}. 
\end{proposition}

In contrast, it is known \cite{Bo} that we can have 
$\Theta\in \mathrm{WEP}$, and $\Theta B^*/B\not\in \mathrm{WEP}$. 

\begin{proof}
Let $a= \sup_{n\ge 1}\rho(z_n,z^*_n) <1$. Fix $0 < \varepsilon < 1$. Since $\Theta$ has the $\mathrm{SIP}$, there exists $t<1$ such that 
$$
\rho(w,\ZTH)\ge t \implies |\Theta(w)|\ge 1-\varepsilon^3. 
$$
Then for some $t_1=t_1(t,a)<1$ we have 
$$
\rho(w,\mathcal Z(\Theta_1))\ge t_1 \implies \rho(w,\ZTH)\ge t \implies |\Theta(w)|\ge 1-\varepsilon^3. 
$$
Furthermore, for some $t_2=t_2(t,a,\varepsilon)<1$, if $\rho(w,\{z^*_n\})\ge t_2$ and $|B(w)|\ge 1-\varepsilon^3,$
 then $|B^*(w)|\ge 1-\varepsilon^2$. Therefore, if $\rho(w,\mathcal Z(\Theta_1))\ge \max(t_1,t_2)$, then 
$$
|\Theta_1(w)|\ge 1-\varepsilon, 
$$
and, hence, $\Theta_1$ has the SIP.
\end{proof}

\subsection{Blaschke products with zeros in a Stolz angle}
\label{abp}

In order to exhibit examples of inner functions in $\mathcal M $ which do not have the SIP, let us consider the 
very restrictive case of Blaschke products $\Theta$ with zeros $\{a_j \}$ within a Stolz angle. Then we always have
$\Theta \in \mathcal M$ \cite[Corollary 4]{MN}. In this particular case, it is known \cite[Theorem 1]{GPV} that the following properties are equivalent: 
\begin{itemize}
\item $\Theta$ is Carleson--Newman;
\item $\{a_j\}$ is a finite union of separated sequences;
\item $\{a_j\}$ is a finite union of exponential sequences.
\end{itemize}
Furthermore, $\Theta$ has the WEP if and only if $\Theta$ is Carleson--Newman \cite[(P7), p. 869]{GMN}.

\begin{proposition}\label{Stolz}
Let $\Theta$ be a Blaschke product  whose zeros $\{a_j\}$ lie in a Stolz angle. Then the following conditions are equivalent
\begin{enumerate}
\item $\Theta$ is Carleson--Newman;
\item $\Theta$ satisfies the {\rm SIP};
\item $\sup \{\eta_\Theta(t) : 0<t<1 \}>0$;
\item $\Theta\in \mathcal{P}$.
\end{enumerate}
\end{proposition}

The Proposition shows that any Blaschke product whose zero set is included in a Stolz angle and is not a finite union of separated sequences belongs to $\mathcal M$ but does not  satisfy the SIP. For instance we may consider the Blaschke product with zeros at the points $1-2^{-j}$ with multiplicity $j$, $j\ge 1$.

\begin{proof}
\noindent (1) $\Longrightarrow (2)$ is a direct consequence of \cite[Theorem 1]{GPV} and Proposition \ref{Corol1}. 

(2) $\Longrightarrow (3)$ is straightforward since the SIP is equivalent to the relation $\lim_{t \to 1} \eta_\Theta(t)=1$.

(3) $\Longrightarrow (1)$: Let $0 < \delta_0 < 1$ be such that $\eps_0=\eta_\Theta(1-\delta_0)>0$ and let $\sigma_0>1$, $\zeta\in \partial \D$,  be such that
\[
\ZTH\subset \{ z \in \D : |z-\zeta|\le \sigma_0 (1-|z|)\}.
\]
Choose $\sigma=\sigma(\sigma_0,\delta_0)$ large enough so that 
\[
|a-\zeta|\le \sigma_0 (1-|a|), |z-\zeta|\ge \sigma (1-|z|)\Longrightarrow \rho(z,a)\ge 1-\delta_0.
\]
For any $a_k\in  \ZTH$, let $z_k$ be such that $|z_k|=|a_k|$ and $|z_k-\zeta|= \sigma (1-|a_k|)$. Then we have $1-\rho(z_k,a_j)\ge c (1- \rho (a_k , a_j))$ for any $k,j,$ where $c>0$ is a constant only depending on $\sigma$ and $\sigma_0$. Then 
\[
\begin{split}
\sum_{j} (1-\rho(a_k,a_j))&\le c^{-1}\sum_{j}(1-\rho(z_k,a_j))\\
\le &-c^{-1}\log |\Theta(z_k)|\le -c^{-1}\log(1-\eps_0).
\end{split} 
\]
Hence $\Theta$ is Carleson--Newman.

(4) $\Longrightarrow (1)$ is obvious.

(1) $\Longrightarrow (4)$: As mentioned before, any Blaschke product whose  zeros are contained in a Stolz angle is in the class $\mathcal{M}$ (\cite[Corollary 4]{MN}). 
By definition, $\mathcal{P}$ consists of the Carleson--Newman Blaschke products in the class $\mathcal{M}$.
\end{proof}







\label{fullblaschke}

\subsection{Analogues for WEP}
\label{anawep}
The methods and results of the previous sections allow us to shed some light
on the previously considered notion of WEP functions \cite{GMN}. 
Let us say that a set is \emph{$R$-narrow} if it contains no hyperbolic disc $D_H(z,R)$.
Then an inner function $\Theta$ is WEP if and only if 
for any $R$, there exists $\eps(R)$ such that the set $\{0<|\Theta|<\eps(R)\}$
is $R$-narrow, while it is a Carleson--Newman Blaschke product if and only if 
for any $R$, there exists $\eps(R)$ such that the larger set $\{|\Theta|<\eps(R)\}$
is $R$-narrow. This last condition is equivalent to being hereditarily WEP, that is to
say, that any factor of $\Theta$ be WEP. This was already proved by P.Gorkin and R.Mortini in \cite[Corollary 4.6]{GM} using properties of the maximal ideal space of $H^\infty$. In this subsection we present a proof using our previous methods.

\begin{lemma}
\label{CN-narrow}
Let $\Theta$ be a Blaschke product with zeros $\{a_j\}$. The following conditions are equivalent:
\begin{enumerate}
\item $\Theta$ is a Carleson--Newman Blaschke product.
\item There exist $r\in(0,1)$ and $\eps\in (0,1)$ such that 
$$
D_\rho(a_k,r)\not\subset\{w: |\Theta(w)|<\eps\}, \qquad k\ge 1.
$$
\item For every $r\in (0,1)$ we have 
$$
\inf_{z\in \D}\frac{1}{2\pi}\int_{-\pi}^\pi \log| \Theta(\varphi_z(re^{i\theta}))|d\theta >-\infty.
$$
Here once again $\varphi_z (w) =\frac{z-w}{1 - w \overline{z}}$, $w \in \D$. 
\item For every $r\in (0,1)$ there exists $\eps\in (0,1)$ such that 
$$
D_\rho(z,r)\not\subset\{w: |\Theta(w)|<\eps\},\qquad z\in \D.
$$
\end{enumerate}
\end{lemma}

 For a similar result see \cite[Theorem 2.2]{MN}. The equivalence between (1) and (2) is due to P.Gorkin and R.Mortini \cite[Theorem 3.6]{GM}.

\begin{proof} (2) $\Longrightarrow$ (1) : 
For every $k\ge 1$, there exists  $w\in D_\rho(a_k,r)$ such that $ |\Theta(w)|\ge \eps$. Then
\begin{multline*}
\sum_j (1-\rho(a_k,a_j))\le 
2\sum_j e^{-2d_H(a_k,a_j)}\le 
2e^{2d_H(a_k,w)}\sum_j e^{-2d_H(w,a_j)}\\ \le
\frac{4}{1-\rho(a_k,w)}\sum_j (1-\rho(w,a_j))\le 
\frac{4}{1-r}\sum_j (1-\rho(w,a_j))\\ \le \frac{4}{1-r}\sum_j (- \log \rho(w,a_j))
\le \frac{-4 \log |\Theta(w) |}{1-r} \le \frac{4 \log (1/\eps)}{1-r}.
\end{multline*}

(1) $\Longrightarrow$ (3) : 
Denote $C=\sup_{z\in \D}\sum_j (1-\rho(z,a_j))$ and let $r\in (0,1)$ and $z\in \D$.  \\
We may assume without loss of generality that $\Theta(z)\not=0$, up to replacing $\Theta$ by $\frac{\Theta}{(\varphi_{z})^m}$ if  $z$ is a zero of $\Theta$ with multiplicity $m$.  Jensen's formula gives that 
\begin{multline*}
\frac{1}{2\pi}\int_{-\pi}^\pi -\log|\Theta(\varphi_z(re^{i\theta}))|d\theta
=-\log|\Theta(z)|-\sum_{a_j\in D_\rho(z,r)}\log \frac{r}{\rho(z,a_j)}
\\
=\sum_{j}-\log \rho(z,a_j)-\sum_{a_j\in D_\rho(z,r)}\log \frac{r}{\rho(z,a_j)}
\\=\sum_{a_j\not\in D_\rho(z,r)}-\log \rho(z,a_j)-\sum_{a_j\in D_\rho(z,r)}\log r
\\
\le \frac{1}{r} \sum_{a_j\not\in D_\rho(z,r)} (1-\rho(z,a_j))-\frac{\log r}{1-r}\sum_{a_j\in D_\rho(z,r)}(1-\rho(z,a_j))
\\
\le \frac{1}{r} \sum_{j} (1-\rho(z,a_j))\le \frac{C}{r} .
\end{multline*}

(3) $\Longrightarrow$ (4) : 
By (3) we find $c=c(r)$ such that  
\[
D_\rho(z,r)\not\subset\{w: |\Theta(w)|<e^{-c}\},\qquad z\in \D.
\]
(4) $\Longrightarrow$ (2) is trivially true.
\end{proof}

\begin{proposition}
\label{WEP}
An inner function $\Theta$ has the {\rm WEP} if and only if for every $r\in (0,1)$ there exists $\eps\in (0,1)$ such that
\begin{equation}\label{wep-narrow}
D_\rho(z,r)\not\subset\{w: 0<|\Theta(w)|<\eps\},\qquad z\in \D.
\end{equation}
\end{proposition}

\begin{proof} If $\Theta$ has the WEP, then for every $r\in (0,1)$ and for every $z\in \D$, either $\ZTH\cap D_\rho(z,r)\not=\emptyset$ 
or $|\Theta(z)|\ge \eta_\Theta(r)>0$, so $\eps= \eta_\Theta(r)$ satisfies \eqref{wep-narrow}. 

Conversely, if $\Theta$ does not have the WEP, then there exists $r>0$
such that for all $\eps>0$,  there exists $z\in \D$ with $\rho(z, \ZTH) \ge r$ and $|\Theta(z)|<\eps^3$.
 On $D_\rho(z,r)$, $-\log|\Theta|$ is a positive harmonic function,
so applying Harnack's inequality, we obtain that for any $w\in D_\rho(z,\frac{r}{2})$, 
\[-\log|\Theta(w)| \ge -\frac{r-\rho(z,w)}{r+\rho(z,w)}\log|\Theta(z)|\ge-\frac{1}{3} \log|\Theta(z)|>  -\log  \eps.\]
So $D_\rho(z,\frac{r}{2})\subset \{w: 0<|B(w)|<\eps\}$ which contradicts (\ref{wep-narrow}).
\end{proof}

Finally we can characterize Carleson--Newman Blaschke products by the hereditary WEP property, as proved in \cite[Corollary 4.6]{GM}.

\begin{proposition}
\label{WEPeq}
Let $\Theta$ be an inner function. Then $\Theta$ is a Carleson--Newman Blaschke product
if and only if the factorization $\Theta=\Theta_1\Theta_2$, with inner functions $\Theta_1, \Theta_2$, implies that  
$\Theta_1$ and $\Theta_2$ have the {\rm WEP}.
\end{proposition}

\begin{proof}
The direct implication is clear, since any factor of a Carleson--Newman Blaschke product
is again Carleson--Newman, thus has the WEP, see \cite{GMN}.

Conversely, if all factors of $\Theta$ have the WEP, then $\Theta$ cannot have a singular factor. 
Let now $\Theta$ be a Blaschke product with zeros $(\lambda_k)_k$, which by the hypothesis has the WEP. 
Suppose that it is not Carleson--Newman.  Then 
$\sup_k \sum_j e^{-2d_H(\lambda_k, \lambda_j)}=\infty$, but each individual sum converges.
So we can choose a sequence $(\lambda_{k_m})_m$ such that $d_H (\lambda_{k_m},\lambda_{k_{m'}})>1$, $m\not= m'$, 
and, letting $S_m:=\sum_j e^{-2d_H(\lambda_{k_m}, \lambda_j)}$,
we have $\lim_{m\to\infty} S_m=\infty$.

For every $z\in D_H(\lambda_{k_m},\frac12)$, by the triangle inequality, we obtain 
\[
\log \frac1{|\Theta(z)|} \ge \sum_j e^{-2d_H(z, \lambda_j)} \ge \frac1e S_m ,
\]
so for every $z\in D_H(\lambda_{k_m},\frac12)$, we have 
$$
d_H (z, \ZTH)\le \tanh^{-1}(\varkappa_\Theta\left( \exp(-S_m/e))\right).
$$
In particular, the number $N_m$ of zeros of $\Theta$ in 
$D_H(\lambda_{k_m},\frac12) \setminus D_H(\lambda_{k_m},\frac14)$ tends to infinity.

Define now $\Theta_1$ to be the Blaschke product with zeroes at 
$\{\lambda_j: d_H(\lambda_j, \lambda_{k_m})\ge \frac14,\, m\ge 1\}$. We have 
$\rho(\lambda_{k_m}, \mathcal Z(\Theta_1)) \ge \tanh \frac14$ for $m\ge 1$, but
$|\Theta_1(\lambda_{k_m})| \le \exp(-N_m \tanh \frac12) \to 0$, which shows that $\Theta_1$ does not have the WEP.
\end{proof}

\subsection{$\eta_\Theta$ could be discontinuous}\label{dis5} Here we give a simple example of a finite Blaschke product $B$ with discontinuous $\eta_B$.

\begin{example} Let $r\in(0,1)$ be sufficiently small, let $\lambda_k=re^{i\pi(2k+1)/4}$, $0\le k\le 3$, and let  
$$
B=\prod_{0\le k\le 3}b_{\lambda_k}.
$$ 
Then $\eta_B$ is not continuous.
\end{example}

\begin{proof} We have
$$
B(z)=\frac{z^4+r^4}{1+r^4z^4},\qquad B(0)=r^4, \qquad |B(z)|\ge \frac{|z|^4-r^4}{1-r^4|z|^4}.
$$ 

Let $\Omega=\{z\in\mathbb D:\rho(z,\mathcal Z(B))>r\}$. Then $\Omega\cap (0,1)=(s,1)$, with
$$
\Bigl| \frac{s-re^{i\pi/4}}{1-sre^{-i\pi/4}} \Bigr| =r, 
$$
that is 
$$
s=\frac{\sqrt{2}r}{1+r^2}.
$$ 

Therefore, 
$$
\Omega\subset\Bigl\{z\in\mathbb D:|z|>\frac{\sqrt{2}r}{1+r^2}\Bigr\},
$$
and 
$$
\inf_{z\in\Omega}|B(z)|\ge  \frac{\frac{4r^4}{(1+r^2)^4}-r^4}{1-r^4\frac{4r^4}{(1+r^2)^4}}=r^4\frac{4-(1+r^2)^4}{(1+r^2)^4-4r^8}.
$$
Thus, for sufficiently small $r>0$ we have
$$
\lim_{s\to r+0}\eta_B(s)\ge 2 r^4.
$$
Since
$$
\{z\in\mathbb D:\rho(z,\mathcal Z(B))\ge r\}=(\overline{\Omega}\cap\D)\cup\{0\},
$$
we obtain that $\eta_B(r)=r^4$ and, thus, $\eta_B$ is not continuous at the point $r$.
\end{proof}

\end{document}